\documentclass[12pt, reqno]{amsart}

\author[J.~Connor]{Jeff Connor}
\address{Department of Mathematics, Ohio University, Athens, OH 45701, USA}
\email{connorj@ohio.edu}

\author[P.~Leonetti]{Paolo Leonetti}
\address{Department of Statistics, Universit\`a Bocconi, via Roentgen 1, Milan 20136, Italy}
\email{leonetti.paolo@gmail.com}


\title{A characterization of $(\mathcal{I}, \mathcal{J})$-regular matrices}

\usepackage[T1]{fontenc}
\usepackage{amsmath}
\usepackage{amssymb}
\usepackage{amsthm}
\usepackage[left=2.5cm, right=2.5cm, top=3cm]{geometry}
\usepackage{hyperref}
\usepackage{fancyhdr}
\usepackage{enumitem}
\usepackage{comment}
\usepackage{nicefrac}
\usepackage{bm}
\usepackage{mathrsfs}
\usepackage{graphicx}
\usepackage[utf8]{inputenc}
\usepackage{cancel}
\usepackage{mathtools}
\usepackage{tikz}
\usetikzlibrary{matrix,arrows,positioning,automata,shapes,shadows,calc,through,intersections}

\makeatletter
\newcommand{\Spvek}[2][r]{%
  \gdef\@VORNE{1}
  \left(\hskip-\arraycolsep%
    \begin{array}{#1}\vekSp@lten{#2}\end{array}%
  \hskip-\arraycolsep\right)}

\def\vekSp@lten#1{\xvekSp@lten#1;vekL@stLine;}
\def\vekL@stLine{vekL@stLine}
\def\xvekSp@lten#1;{\def\temp{#1}%
  \ifx\temp\vekL@stLine
  \else
    \ifnum\@VORNE=1\gdef\@VORNE{0}
    \else\@arraycr\fi%
    #1%
    \expandafter\xvekSp@lten
  \fi}
\makeatother

\AtBeginDocument{%
   \def\MR#1{}
}

\newtheorem{thm}{Theorem}[section]
\newtheorem{cor}[thm]{Corollary}
\newtheorem{lem}[thm]{Lemma}

\theoremstyle{definition} 
\newtheorem{defi}[thm]{Definition}
\let\olddefi\defi
\renewcommand{\defi}{\olddefi\normalfont}
\newtheorem{example}[thm]{Example}
\let\oldexample\example
\renewcommand{\example}{\oldexample\normalfont}
\newtheorem{rmk}[thm]{Remark}
\let\oldrmk\rmk
\renewcommand{\rmk}{\oldrmk\normalfont}

\hypersetup{
    pdftitle={A characterization of ideal regular matrices},
    pdfauthor={Jeff Connor and Paolo Leonetti},
    pdfmenubar=false,
    pdffitwindow=true,
    pdfstartview=FitH,
    colorlinks=true,
    linkcolor=blue,
    citecolor=green,
    urlcolor=cyan
}

\providecommand{\MR}[1]{}

\providecommand{\MR}{\relax\ifhmode\unskip\space\fi MR }

\providecommand{\href}[2]{#2}

\begin{document}

\maketitle
\thispagestyle{empty}

\begin{abstract}
Let $\mathcal{I},\mathcal{J}$ be two ideals on $\mathbf{N}$ which contain the family $\mathrm{Fin}$ of finite sets. We provide necessary and sufficient conditions on the entries of an infinite real matrix $A=(a_{n,k})$ which maps $\mathcal{I}$-convergent bounded sequences into $\mathcal{J}$-convergent bounded sequences and preserves the corresponding ideal limits. The well-known characterization of regular matrices due to Silverman--Toeplitz corresponds to the case $\mathcal{I}=\mathcal{J}=\mathrm{Fin}$. 

Lastly, we provide some applications to permutation and diagonal matrices, which extend several known results in the literature.
\end{abstract}


\section{Introduction}

Let $\omega$ be the set of all real sequences indexed by the positive integers $\mathbf{N}$, and let $A=(a_{n,k})$ be an infinite real matrix. 
A sequence $x \in \omega$ is said to be $A$-summable if the sequence $Ax=\left(\sum_k a_{n,k}x_k\right)$ is well defined and it is convergent, i.e., $Ax \in c$. 
The matrix $A$ is called \textbf{regular} if every convergent sequence is $A$-summable and it preserves the limit, i.e., $\lim Ax=\lim x$ for all $x \in c$. 
A classical result of Silverman--Toeplitz characterizes 
the class of regular matrices: 
\begin{thm}\label{thm:silvermaintoeplitz}
A matrix 
$A$ is 
regular
if and only if\textup{:}
\begin{enumerate}[label={\rm (\textsc{S}\arabic{*})}]
\item \label{item:S1} $\sup_n \sum_k |a_{n,k}| < \infty$\textup{;}
\item \label{item:S3} $\lim_n a_{n,k}=0$ for each $k$\textup{;}
\item \label{item:S2} $\lim_n \sum_k a_{n,k}=1$\textup{.}
\end{enumerate}
\end{thm}

The aim of this work is to generalize Theorem \ref{thm:silvermaintoeplitz} in the context of ideal convergence. 

Recall that an ideal $\mathcal{I}\subseteq \mathcal{P}(\mathbf{N})$ is a family of subsets of $\mathbf{N}$ closed under taking finite unions and subsets. 
Unless otherwise stated, it is also assumed that $\mathcal{I}$ is admissible, i.e., it contains the family of finite sets $\mathrm{Fin}$ and $\mathcal{I}\neq \mathcal{P}(\mathbf{N})$. Let $\mathcal{I}^\star=\{A\subseteq \mathbf{N}: A^c \in \mathcal{I}\}$ be its dual filter. 
An important example of ideal is the family of asymptotic density zero sets, that is, 
$$
\textstyle \mathcal{Z}=\left\{S\subseteq \mathbf{N}: \lim_{n\to \infty} \frac{1}{n}|S\cap [1,n]|=0\right\}.
$$

A sequence $x\in \omega$ is $\mathcal{I}$-convergent to $\eta$, shortened as $\mathcal{I}\text{-}\lim x=\eta$, if $\{n \in \mathbf{N}: |x_n-\eta|>\varepsilon\} \in \mathcal{I}$ for all $\varepsilon>0$; 
$\mathcal{Z}$-convergence is often referred to as statistical convergence in the summability theory literature. 
We let $c(\mathcal{I})$ be the vector space of $\mathcal{I}$-convergent sequences and $c_0(\mathcal{I})$ be its subspace of sequences with $\mathcal{I}$-limit $0$. 
Structural properties of the set of bounded $\mathcal{I}$-convergent sequences $c(\mathcal{I}) \cap \ell_\infty$ and its subspace $c_0(\mathcal{I}) \cap \ell_\infty$ have been recently studied in the literature, sometimes providing answers to longstanding questions, see e.g. \cite{MR2735533, 
MR3883309, MR3836186, Miek}. 

Given sequence spaces $X,Y\subseteq \omega$, we let $(X,Y)$ be the set of infinite real matrices $A$ such that $Ax$ is well defined and belongs to $Y$ for all $x \in X$. Accordingly, a matrix $A$ is regular if $A \in (c,c)$ and preserves the (ordinary) limits. 
The relationship between ideal convergence and matrix summability has been recently studied in \cite{MR3863065}. 

With these premises, we can now state the main definition of this work:   
Given ideals $\mathcal{I}, \mathcal{J}$ on $\mathbf{N}$, a matrix $A$ is said to be $\bm{(\mathcal{I}, \mathcal{J})}$\textbf{-regular} if $A \in (c(\mathcal{I}) \cap \ell_\infty, c(\mathcal{J}) \cap \ell_\infty)$ and 
$$
\forall x \in c(\mathcal{I}) \cap \ell_\infty,\quad 
\mathcal{I}\text{-}\lim x=\mathcal{J}\text{-}\lim Ax.
$$

Considering that $c=c(\mathrm{Fin})\cap \ell_\infty$, Theorem \ref{thm:silvermaintoeplitz} amounts to a characterize the class of $(\mathrm{Fin},\mathrm{Fin})$-regular matrices. 
%
\begin{thm}\label{thm:mainsilvermantoeplitz} 
Let $\mathcal{I},\mathcal{J}$ be ideals on $\mathbf{N}$. Then a matrix $A$ is $(\mathcal{I}, \mathcal{J})$-regular provided that 
\begin{enumerate}[label={\rm (\textsc{T}\arabic{*})}]
\item \label{item:T1} $\sup_n \sum_k |a_{n,k}| < \infty$\textup{;}
\item \label{item:T2} $\mathcal{J}\text{-}\lim_n \sum_k a_{n,k}=1$\textup{;}
\item \label{item:T3} $\mathcal{J}\text{-}\lim_n \sum_{k \in E} |a_{n,k}|=0$ for all $E \in \mathcal{I}$\textup{.}
\end{enumerate}
Conversely, if $A$ is $(\mathcal{I}, \mathcal{J})$-regular, then it satisfies \ref{item:T1} and \ref{item:T2}.
\end{thm}

A nontrivial problem is to characterize when conditions \ref{item:T1}--\ref{item:T3}, which depend only on the entries of $A$, provide a full characterization of $(\mathcal{I}, \mathcal{J})$-regularity, that is, equivalently, when an $(\mathcal{I}, \mathcal{J})$-regular matrix satisfies condition \ref{item:T3}.

On the positive side, we have the following result.
\begin{thm}\label{thm:positive}
Let $\mathcal{I}, \mathcal{J}$ be ideals on $\mathbf{N}$ such that $A$ is nonnegative or $\mathcal{I}=\mathrm{Fin}$ or $\mathcal{J}=\mathrm{Fin}$. 
Then $A$ is $(\mathcal{I}, \mathcal{J})$-regular if and only if it satisfies \ref{item:T1}--\ref{item:T3}. 
\end{thm}
\noindent This provides a generalization to \cite[Theorem 3.1]{MR3511151}, which studies the case where $A$ is nonnegative and, in addition, it belongs to $(\ell_\infty, \ell_\infty)$, 
cf. Theorem \ref{lem:boundedtobounded} below. If $A$ is nonnegative, note that conditions \ref{item:T2}--\ref{item:T3} may be simplied together into
\begin{enumerate}[label={\rm (\textsc{T4})}]
\item \label{item:T4} $\mathcal{J}\text{-}\lim_n \sum_{k \in I} a_{n,k}=1$ for all $I \in \mathcal{I}^\star$.
\end{enumerate}

On the negative side, it turns out that an $(\mathcal{I}, \mathcal{J})$-regular matrix does not necessarily satisfy condition \ref{item:T3}. To the best of authors' knowledge, this is the first result of this type. 
\begin{thm}\label{thm:negativemain}
Let $\mathcal{I}$ be an ideal on $\mathbf{N}$ such that $\mathrm{Fin}\subsetneq \mathcal{I}\subseteq \mathcal{Z}$. Then there exists an $(\mathcal{I}, \mathcal{Z})$-regular matrix $A$ which does not satisfy \ref{item:T3}.
\end{thm}

A remarkable example of ideal $\mathcal{I}$ which satisfies the hypotheses of Theorem \ref{thm:negativemain} is the ideal of uniform density zero sets $\{S\subseteq \mathbf{N}: \lim_{n\to \infty} \max_{k\ge 0}\frac{1}{n}|S\cap [k+1,k+n]|=0\}$, cf. \cite{MR2835960}.

Lastly, note that Theorem \ref{thm:mainsilvermantoeplitz} is also analogous to \cite[Theorem 5.1]{MR2966607}, which studies the case  where both $\mathcal{I}$ and $\mathcal{J}$ are assumed to be 
ideals generated by nonnegative regular matrices. 
However, it is known by \cite[Proposition 13]{MR3405547} that, if we regard ideals as subsets of the Cantor space $\{0,1\}^{\mathbf{N}}$ endowed with the product topology, then the ideal 
$$
\textstyle 
\mathcal{I}_A:=\{S\subseteq \mathbf{N}: \lim_n \sum_{k \in S}a_{n,k}=0\}
$$
generated by a regular matrix $A$ is necessarily an $F_{\sigma\delta}$-ideal, 
and hence \cite[Theorem 5.1]{MR2966607} only applies to a restricted class of ideals. 

Several other applications for permutations and diagonal matrices have been studied in the literature \cite{Connor2020, 
MR1955734, 
MR2003000}, 
see Section \ref{sec:applications} below. 

Related results which characterize matrix classes of the type $(c(\mathcal{I}) \cap \ell_\infty, c(\mathcal{J}) \cap \ell_\infty)$ (hence, without preserving necessarily their ideal limits) can be found e.g. in \cite{MR1245744, MR2607117, MR0364938, MR1433948}. 

\section{Main Proofs}\label{sec:mainproof}

Hereafter, $\ell_\infty$ and any of its subspaces are endowed with the supremum norm.
The following result is well known, see e.g. 
\cite[Theorem 2.3.5]{MR1817226}:
\begin{thm}\label{lem:boundedtobounded}
$A \in (\ell_\infty, \ell_\infty)$ if and only if 
$A \in (c_0, \ell_\infty)$ if and only if 
$A$ satisfies \ref{item:T1}.
\end{thm}

The next result which will be the key tool in the proof 
of our main result. 
Related results can be found in \cite[Proposition 6.3]{MR3671266} (for the case of $F_\sigma$-ideals $\mathcal{I}, \mathcal{J}$) and in \cite[Theorem 3.7]{MR954458} (for the case $\mathcal{I}=\mathcal{Z}$ and $\mathcal{J}=\mathrm{Fin}$).
\begin{lem}\label{thm:mainc0c0}
Let $\mathcal{I},\mathcal{J}$ be ideals on $\mathbf{N}$. 
If a matrix $A$ 
satisfies \ref{item:T1} and \ref{item:T3} then $A \in (c_0(\mathcal{I}) \cap \ell_\infty, c_0(\mathcal{J})\cap \ell_\infty)$, which in turn implies that $A$ satisfies \ref{item:T1}. 
\end{lem}
\begin{proof}
First, let us prove that if $A$ satisfies \ref{item:T1} and \ref{item:T3} then $A \in (c_0(\mathcal{I}) \cap \ell_\infty, c_0(\mathcal{J})\cap \ell_\infty)$. Fix $x \in c_0(\mathcal{I}) \cap \ell_\infty$. By condition \ref{item:T1} and Theorem \ref{lem:boundedtobounded}, we get $A \in (\ell_\infty,\ell_\infty)$, hence 
$Ax$ is a well-defined bounded sequence. 
At this point, fix $\varepsilon>0$ and define
$$
\delta:=\frac{\varepsilon}{2\max\{1,M,\|x\|\}}>0,
$$
where $M:=\sup_n \sum_k |a_{n,k}|$ and $\|x\|=\sup_k |x_k|$. 
Note that, since $x \in c_0(\mathcal{I})$ then 
$ 
K:=\{n\in \mathbf{N}: |x_n|>\delta\}\in \mathcal{I}.
$ 
By condition \ref{item:T3}, 
it follows that 
$
\textstyle S:=\{n \in \mathbf{N}: \sum_{k \in K}|a_{n,k}|>\delta\} \in \mathcal{J}.
$ 
Now observe that if $\sum_k |a_{n,k}x_k|>\varepsilon$ then 
$$
\textstyle 
\sum_{k \in K} |a_{n,k}x_k|>\frac{\varepsilon}{2}
\quad \text{or}\quad 
\sum_{k \notin K} |a_{n,k}x_k|>\frac{\varepsilon}{2}.
$$
If it is the case that $\sum_{k \in K} |a_{n,k}x_k|>\frac{\varepsilon}{2}$ then
$$
\textstyle 
\frac{\varepsilon}{2}<\sum_{k \in K} |a_{n,k}| |x_k|\le \|x\|\sum_{k \in K} |a_{n,k}|
$$
and hence $\delta \le \varepsilon/2\max\{1,\|x\|\}<\sum_{k \in K} |a_{n,k}|$, i.e., $n\in S$. 
In the event that $\sum_{k \notin K} |a_{n,k}x_k|>\frac{\varepsilon}{2}$, it follows that 
$$
\textstyle 
\frac{\varepsilon}{2}<\sum_{k \notin K} |a_{n,k}| |x_k|\le \delta \sum_{k \notin K} |a_{n,k}|\le \delta M\le  \frac{\varepsilon}{2},
$$
which is a contradiction, hence 
$\{n \in \mathbf{N}: |\sum_k a_{n,k}x_k|>\varepsilon\}\subseteq S \in \mathcal{J}$. Since $\varepsilon$ is arbitrary, we conclude that $\mathcal{J}\text{-}\lim Ax=0$.

\bigskip

To prove the second part of the statement, it is sufficient to observe that 
$$
(c_0(\mathcal{I}) \cap \ell_\infty, c_0(\mathcal{J})\cap \ell_\infty)\subseteq (c_0,\ell_\infty)
$$
and to use Theorem \ref{lem:boundedtobounded}.
\end{proof}

\begin{lem}\label{lem:c0cc}
Let $\mathcal{I}, \mathcal{J}$ be ideals on $\mathbf{N}$. Then $A$ is $(\mathcal{I}, \mathcal{J})$-regular if and only if $A \in (c_0(\mathcal{I}) \cap \ell_\infty, c_0(\mathcal{J})\cap \ell_\infty)$ and \ref{item:T2} holds.
\end{lem}
\begin{proof}
\textsc{If Part.} Suppose that $A \in (c_0(\mathcal{I})\cap \ell_\infty, c_0(\mathcal{J})\cap \ell_\infty)$ and satisfies condition \ref{item:T2}. Fix a sequence $x \in c(\mathcal{I}) \cap \ell_\infty$ such that $\mathcal{I}\text{-}\lim x=\eta$. 
It follows that the sequence $y$ defined by $y_n:=x_n-\eta$ for all $n$ belongs to $c_0(\mathcal{I})\cap \ell_\infty$, so that $Ay \in c_0(\mathcal{J}) \cap \ell_\infty$. 
Therefore
$$
\textstyle 0=\mathcal{J}\text{-}\lim Ay=\mathcal{J}\text{-}\lim Ax-\eta \cdot \mathcal{J}\text{-}\lim_n \sum_k a_{n,k},
$$
which implies that $Ax \in c(\mathcal{J}) \cap \ell_\infty$ and $\mathcal{J}\text{-}\lim Ax=\eta$. 

\medskip

\textsc{Only If Part.} Suppose that $A$ is a $(\mathcal{I},\mathcal{J})$-regular matrix. In particular, it is clear that $A \in (c_0(\mathcal{I})\cap \ell_\infty, c_0(\mathcal{J})\cap \ell_\infty)$. Finally, let $x$ be the constant sequence $1$. Then $\lim_n x_n=1$, so that  $Ax=(\sum_{k}a_{n,k}) \in c(\mathcal{J}) \cap \ell_\infty$ and $\mathcal{J}\text{-}\lim_n \sum_{k}a_{n,k}=1$, i.e., condition \ref{item:T2} holds.
\end{proof}

\begin{proof}
[Proof of Theorem \ref{thm:mainsilvermantoeplitz}]
It follows putting together Lemma \ref{thm:mainc0c0} and Lemma \ref{lem:c0cc}.
\end{proof}

\begin{proof}
[Proof of Theorem \ref{thm:positive}] 
Thanks to Theorem \ref{thm:mainsilvermantoeplitz}, it sufficient to show that, in each case, an $(\mathcal{I}, \mathcal{J})$-regular matrix satisfies condition \ref{item:T3}. 
Note that if $A$ is $(\mathcal{I}, \mathcal{J})$-regular, then $\mathcal{J}\text{-}\lim Ae^k=0$ by condition \ref{item:T2}, where $e^k$ is the $k$th unit vector of $\ell_\infty$ for each $k \in \mathbf{N}$, i.e., $\mathcal{J}\text{-}\lim_n |a_{n,k}|=0$. This implies that 
\begin{equation}\label{eq:sliding}
\textstyle 
\forall m \in \mathbf{N}, \quad 
\lim_n \sum_{k\le m}|a_{n,k}|=0.
\end{equation}

\medskip

First, suppose that $A$ is a nonnegative $(\mathcal{I}, \mathcal{J})$-regular matrix and fix $E \in \mathcal{I}$. Since $\bm{1}_E \in c_0(\mathcal{I})\cap \ell_\infty$, we obtain that $\mathcal{J}\text{-}\lim A\bm{1}_E=\mathcal{J}\text{-}\lim_n \sum_{k\in E}a_{n,k}=0$. 

\medskip

Second, thanks to \eqref{eq:sliding}, condition \ref{item:T3} holds if $\mathcal{I}=\mathrm{Fin}$.

\medskip

Third, let $A$ be an $(\mathcal{I}, \mathrm{Fin})$-regular matrix 
with $\mathcal{I}\neq \mathrm{Fin}$, and hence $\mathcal{J}=\mathrm{Fin}$ and  $\mathcal{J}\text{-}\lim$ is the ordinary limit.
Suppose for the sake of contradiction that condition \ref{item:T3} fails. 
Hence there exists an infinite set $I \in \mathcal{I}$ such that 
\begin{equation}\label{eq:contradictionE}
\textstyle \limsup_{n} \sum_{k \in I}|a_{n,k}| \neq 0.
\end{equation}
Accordingly, let $(i_n)$ be the increasing enumeration of $I$. 
Taking into account \eqref{eq:contradictionE} and the fact that $ \sum_{k\in I}|a_{n,k}| \le \sum_k |a_{n,k}| \le \sup_i \sum_i |a_{i,k}|$ for all $n\in \mathbf{N}$, it follows that the bounded sequence $(\sum_{k\in I}|a_{n,k}|)$ has at least one accumulation point, let us say $\kappa \in [0,\sup_n \sum_k |a_{n,k}|]$, that is, $\kappa$ verifies $\{n \in \mathbf{N}: |\sum_{k\in I}|a_{n,k}|-\kappa|<\varepsilon\} \notin \mathrm{Fin}$ for all $\varepsilon>0$. 
In addition, thanks to \eqref{eq:contradictionE}, we may assume that $\kappa \neq 0$. 
This implies that
$$
\textstyle S:=\left\{n \in \mathbf{N}: \frac{7}{8}\kappa < \sum_{k \in I}|a_{n,k}| < \frac{9}{8}\kappa\right\}\notin \mathrm{Fin}.
$$
With these premises, we are going to construct 
a sequence $x \in \{-1,0,1\}^{\mathbf{N}}$ supported on $I$ (hence, in particular, $x \in c_0(\mathcal{I}) \cap \ell_\infty$) such that the sequence $Ax$ is (well defined and) not convergent to $0$. This would provide the desired contradiction. 

Thus, we define two increasing sequences $(s_n)$ and $(m_n)$ of positive integers as it follows. Fix arbitrarily $s_1\in S$ and choose $m_1 \in \mathbf{N}$ such that $\sum_{j\le m_1}|a_{s_1,i_j}|\ge \frac{7}{8}\kappa$. At this point, suppose that $s_1,\ldots,s_{n-1}$ and $m_1,\ldots,m_{n-1}$ have been defined. Then choose $s_n$ and $m_n$ recursively such that:
\begin{enumerate}[label=(\roman*)]
\item \label{item:condition1} $s_n\in S\setminus [1,s_{n-1}]$ and $\sum_{j\le m_{n-1}}|a_{s_n,i_j}| \le \frac{1}{8}\kappa$ (which is possible, thanks to \eqref{eq:sliding});
\item \label{item:condition2} $m_n>m_{n-1}$ and $\sum_{j\le m_n}|a_{s_n,i_j}| \ge \frac{7}{8}\kappa$.
\end{enumerate}
Thus, for each $n \in \mathbf{N}$, define
$$
\textstyle 
\alpha_{n}:=\sum_{j \le m_{n-1}}|a_{s_n,i_j}|, 
\quad
\beta_{n}:=\sum_{m_{n-1}<j \le m_n}|a_{s_n,i_j}|, 
\quad \text{ and }\quad
\gamma_{n}:=\sum_{j>m_{n}}|a_{s_n,i_j}|.
$$
According to points \ref{item:condition1} and \ref{item:condition2} above and the definition of $S$, we have 
\begin{equation}\label{eq:inequalitiesalphabetagamma}
\textstyle
\forall n \in \mathbf{N}, \quad
\alpha_n \le \frac{1}{8}\kappa, 
\quad 
\alpha_n+\beta_n \ge \frac{7}{8}\kappa,
\quad \text{and} \quad 
\frac{7}{8}\kappa< 
\alpha_n+\beta_n+\gamma_n < \frac{9}{8}\kappa. 
\end{equation}

To conclude, let $x=(x_n)$ be the sequence supported on $I$ such that $x_{i_j}=1$ if there exists $n \in \mathbf{N}$ for which $j \in (m_{n-1},m_n]$ and $a_{s_n,i_j}>0$, where $m_0:=1$; otherwise $x_{i_j}=-1$. It follows by construction that 
\begin{displaymath}
\textstyle 
(Ax)_{s_n}
=\sum_{i \in I} a_{s_n,i}x_{i}
=\sum_{j\le m_{n-1}} a_{s_n,i_j}x_{i_j}
+\sum_{m_{n-1}<j \le m_n}|a_{s_n,i_j}|
+\sum_{j>m_n} a_{s_n,i_j}x_{i_j}
\end{displaymath}
for all $n \in \mathbf{N}$. Hence, thanks to \eqref{eq:inequalitiesalphabetagamma}, we obtain 
\begin{displaymath}
\begin{split}
\textstyle 
\forall n \in \mathbf{N},\quad 
|(Ax)_{s_n}|
&\textstyle \ge \beta_n-\alpha_n-\gamma_n\\
&\textstyle =2(\alpha_n+\beta_n)-2\alpha_n-(\alpha_n+\beta_n+\gamma_n)>\frac{3}{8}\kappa.
\end{split}
\end{displaymath}
Hence 
$\{n \in \mathbf{N}: |(Ax)_n|>\frac{3}{8}\kappa\}$ 
contains the infinite set $\{s_n: n \in \mathbf{N}\}\notin\mathcal{J}=\mathrm{Fin}$, which 
implies that $\lim Ax\neq 0$. 
This contradiction concludes the proof. 
\end{proof}


\begin{proof}[Proof of Theorem \ref{thm:negativemain}]
The proof has two main steps. First we construct a matrix $A \in (\ell_\infty, c_0(\mathcal{Z})\cap \ell_\infty$) which fails \ref{item:T3}, and then, using $A$, construct a matrix $B$ that is $(\mathcal{I}, \mathcal{Z})$-regular and fails \ref{item:T3}. 

For each $m \in \mathbf{N}$, let $\lambda_m$ be the smallest nonnegative integer $t$ such that $2^{t} \ge m!$, so that 
$$
\textstyle 
\forall m\in \mathbf{N}, \quad 
m! \le 2^{\lambda_m}=2^{\lceil \log_2 m!\rceil}\le 2\cdot m!.
$$
Set $R:=\bigcup_m R_m$, where $R_m:=\{m!,m!+1,\ldots, m!+2^{\lambda_m}-1\}$ for all $m \in \mathbf{N}$. 
In particular, $R_1=\{1,2\}$, $R_2=\{2,3\}$, $R_3=\{6,7,\ldots,13\}$, etc., so that $\max R_m<\min R_{m+1}$ for all $m\ge 2$. 
By hypothesis $\mathcal{I}\neq \mathrm{Fin}$, hence there exists an infinite set $I \in \mathcal{I}\setminus \mathrm{Fin}$. Let $(i_j)$ be the increasing enumeration of the elements of $I$ and define
$$
\textstyle 
\forall m \in \mathbf{N}, \quad 
Q_m:=R_m \times C_m,
$$
where $C_m:=\{i_{\alpha_{m-1}+1},i_{\alpha_{m-1}+2},\ldots,i_{\alpha_m}\}$, $\alpha_m:=\sum_{i\le m}\lambda_i$, and, by convention, $\alpha_0:=0$.

At this point, let us define the matrix $A$ such that
\begin{displaymath}
\forall n,k \in \mathbf{N}, \quad 
|a_{n,k}|=
\begin{cases}
\,1/m \,\,& \text{ if there exists } m \in \mathbf{N} \text{ such that }(n,k) \in Q_m,\\
\,0 & \text{ otherwise},
\end{cases}
\end{displaymath}
\noindent and, in addition, for all $m \in \mathbf{N}$, the vectors of signs of each row of $Q_m$
$$
\textstyle 
(\mathrm{sgn}(a_{n,i_{\alpha_{m-1}+1}}), \mathrm{sgn}(a_{n,i_{\alpha_{m-1}+2}}), \ldots, \mathrm{sgn}(a_{n,i_{\alpha_{m}}}))
$$
are all distinct (note this is possible; here $\mathrm{sgn}(z):=1$ if $z>0$ and $\mathrm{sgn}(z)=-1$ if $z<0$). 

Now, we claim that $A \in (\ell_\infty, c_0(\mathcal{Z})\cap \ell_\infty)$. Let $x$ be a nonzero bounded sequence. Then $Ax$ is the bounded sequence such that $(Ax)_n$ is equal to $0$ if $n\notin R$ and $\frac{1}{m}\sum_{k}\mathrm{sgn}(a_{n,k})x_k$ for all $n \in R_m$ and $m \in \mathbf{N}$. For each $\varepsilon>0$, it follows that
\begin{equation}\label{eq:unionssiugsdjhg}
\{n\in \mathbf{N}: |(Ax)_n|>\varepsilon\} 
\subseteq R_1\cup \bigcup_{m\ge 2}\left\{n \in R_m: \frac{1}{m}\left|\sum_{k \in C_m}\mathrm{sgn}(a_{n,k})x_k\right|>\varepsilon\right\}.
\end{equation}
In addition, by the weak law of large numbers we have that
$$
\lim_{m \to \infty}\frac{1}{2^{\lambda_m}}\left|\left\{n \in R_m: \frac{1}{m}\left|\sum_{k \in C_m}\mathrm{sgn}(a_{n,k})x_k\right|>\varepsilon\right\}\right|=0,
$$ 
cf. e.g. \cite[Theorem 6.2]{MR830424}. Hence the upper asymptotic density (which is the function $\mathsf{d}^\star_g$ defined below in \eqref{eq:upperdensity} with $g(n)=n$ for all $n$) of the latter union in \eqref{eq:unionssiugsdjhg} is at most
$$
\limsup_{m\to \infty}\frac{\max R_{m-1}+o(2^{\lambda_m})}{\min R_m}\le \limsup_{m\to \infty}\frac{2\cdot (m-1)!+o(m!)}{m!}=0. 
$$
Therefore $\mathcal{Z}\text{-}\lim Ax=0$. 

Observe that $R\notin \mathcal{Z}$, indeed its upper asymptotic is at least
$$
\limsup_{m\to \infty}\frac{|R \cap [1, \max R_m]|}{\max R_m} \ge 
\limsup_{m \to \infty}\frac{|R_m|}{m!+|R_m|}\ge 
\limsup_{m \to \infty}\frac{m!}{m!+2\cdot m!-1}= 
\frac{1}{3}.
$$
Therefore $\mathcal{Z}\text{-}\lim_n \sum_{k \in I}|a_{n,k}|\neq 0$, i.e., $A$ does not satisfy \ref{item:T3}. However, with a similar reasoning, the matrix $A$ does not satisfy \ref{item:T2} as well. Hence, we need to modify slightly the definition of $A$ to conclude the proof.

For, define $B=(b_{n,k}):=A+\mathrm{Id}$, where $\mathrm{Id}$ stands for the (infinite) identity matrix. 
Then $B \in (\ell_\infty, \ell_\infty)$ since $\sup_n \sum_k |b_{n,k}|\le 1+\sup_n \sum_k |a_{n,k}|<\infty$. Fix a sequence $x \in c(\mathcal{I}) \cap \ell_\infty$ with $\eta:=\mathcal{I}\text{-}\lim x$. 
Thus $Bx$ is a well-defined bounded sequence. In addition, since $A \in (\ell_\infty, c_0(\mathcal{Z}) \cap \ell_\infty)$ and $\mathcal{I}\subseteq \mathcal{Z}$, we have $$
\mathcal{Z}\text{-}\lim Bx=\mathcal{Z}\text{-}\lim Ax+\mathcal{Z}\text{-}\lim x=\eta.
$$
Therefore $B$ is $(\mathcal{I}, \mathcal{Z})$-regular (in particular, differently from $A$, the matrix $B$ satisfies \ref{item:T2}). Lastly, note that it follows by construction that
$$
\sum_{k \in I}|b_{n,k}| \ge \left(1-\frac{1}{m}\right)+(m-1)\cdot \frac{1}{m}
\ge 1
$$
for all integers $m\ge 2$ and $n \in R_m$. 
Since $R\notin \mathcal{Z}$ (hence also $R\setminus R_1 \notin \mathcal{Z}$), we conclude that 
$\mathcal{Z}\text{-}\lim_n \sum_{k \in I}|b_{n,k}|\neq 0$. 
Therefore $B$ does not satisfy \ref{item:T3}. 
\end{proof}

Incidentally, note that matrix $A$ defined above belongs to $(\ell_\infty, c_0(\mathcal{Z})\cap \ell_\infty)$, hence also to $(c_0(\mathcal{I}) \cap \ell_\infty, c_0(\mathcal{Z}) \cap \ell_\infty)$. This shows that the reverse implication of the first part of the statement of Lemma \ref{thm:mainc0c0} does not hold as well.

\section{Applications}\label{sec:applications}

\subsection{Diagonal matrices} Given a sequence $s \in \omega$, we denote by $\mathrm{diag}(s)$ the diagonal matrix $A_s=(a_{n,k})$ defined by $a_{n,k}=0$ and $a_{n,n}=s_n$ for all distinct $n,k \in \mathbf{N}$. 
Accordingly, given sequence space $X,Y\subseteq \omega$, we let $m(X,Y)$ be the set of so-called \textbf{multipliers} from $X$ into $Y$, i.e., the set of all sequences $s \in \omega$ such that $sx:=(s_nx_n)$ is a well-defined sequence in $Y$ for all $x \in X$, cf. \cite{MR1955734}. 
In other words, 
$$
\textstyle 
m(X,Y):=\{s \in \omega: \mathrm{diag}(s) \in (X,Y)\}.
$$
\begin{thm}
Let $\mathcal{I},\mathcal{J}$ be ideals on $\mathbf{N}$. Then
$$
\textstyle 
m(c_0(\mathcal{I}) \cap \ell_\infty, c_0(\mathcal{J})\cap \ell_\infty) =\{s \in \ell_\infty: \mathcal{J}\text{-}\lim s\bm{1}_E=0 \text{ for all }E \in \mathcal{I}\}.
$$
\end{thm}
\begin{proof}
First suppose that $s\in\ell_\infty$ and $\mathcal{J}\text{-}\lim s\bm{1}_E=0 \text{ for all }E \in \mathcal{I}$. Then $s\in\ell_\infty$ yields that $\mathrm{diag}(s)$ satisfies \ref{item:T1} and, as $\mathcal{J}\text{-}\lim s\bm{1}_E=0$ implies $\mathcal{J}\text{-}\lim |s|\bm{1}_E=0$, \ref{item:T3} is satisfied and hence $s\in m(c_0(\mathcal{I}) \cap \ell_\infty, c_0(\mathcal{J})\cap \ell_\infty)$ as $\mathrm{diag}(s)\in(c_0(\mathcal{I})\cap\ell_\infty,(c_0(\mathcal{J})\cap\ell_\infty)$. 
Conversely, $\mathrm{diag}(s)\in(c_0(\mathcal{I})\cap\ell_\infty,(c_0(\mathcal{J})\cap\ell_\infty)$ yields $s\in\ell_\infty$, by Lemma \ref{thm:mainc0c0}, and  $\mathcal{J}\text{-}\lim s\bm{1}_E=0 \text{ for all }E \in \mathcal{I}$ by the definition of a multiplier. 
\end{proof}

At this point, note that, for each $s \in \omega$, the sequence $s\bm{1}_E$ is supported on $E$. In particular $\mathcal{J}\text{-}\lim s\bm{1}_E=0$ whenever $E \in \mathcal{J}$. This implies that:
\begin{cor}\label{cor:multiplierinclusion}
$m(c_0(\mathcal{I}) \cap \ell_\infty, c_0(\mathcal{J})\cap \ell_\infty) =\ell_\infty$ for all ideals $\mathcal{I}\subseteq \mathcal{J}$. 
\end{cor}
Particular instances of Corollary \ref{cor:multiplierinclusion} have been already obtained in \cite[Theorem 3]{MR2003000} where $\mathcal{I}=\mathcal{J}=\mathcal{I}_A$ for some nonnegative regular matrix $A$ and \cite[Theorem 1 and Theorem 3]{MR1955734} where $\mathcal{I}=\mathcal{J}$ or $\mathcal{I}=\mathrm{Fin}$.

\subsection{Permutations} 
Given a permutation $\sigma$ of $\mathbf{N}$ and a sequence $x \in \omega$, we let $\sigma(x)$ be the sequence defined by 
$\sigma(x)_n:=x_{\sigma^{-1}(n)}$ for all $n \in \mathbf{N}$, cf. e.g.  \cite{MR1136599, MR966093}. 
Equivalently, 
$$
\textstyle \forall x \in \omega, \quad \sigma(x)=A_\sigma x,
$$
where $A_\sigma=(a_{n,k})$ is the matrix defined by $a_{n,k}=1$ if $\sigma^{-1}(n)=k$ and $a_{n,k}=0$ otherwise. 
Accordingly, given sequence spaces $X,Y\subseteq \omega$, we write $\sigma \in (X,Y)$ to mean $A_\sigma \in (X,Y)$ (and similarly for $(\mathcal{I},\mathcal{J})$-regularity). 
In addition, let $\widehat{\sigma}$ be 
the sequence defined by 
$$
\textstyle 
\forall n \in \mathbf{N}, \quad 
\widehat{\sigma}_n:=n/\sigma^{-1}(n). 
$$
Note that $\widehat{\sigma}$ already appeared in the literature: indeed, it has been shown in \cite{MR2464850} that a permutation $\sigma$ belongs to the L\'{e}vy group, i.e., $\sigma$ satisfies $\lim_n \frac{1}{n}|\{k \in [1,n]: \sigma(k)>n\}|=0$, if and only if $\mathcal{Z}\text{-}\lim \widehat{\sigma^{-1}}=1$; see also \cite[Theorem 2.3]{MR2464850} for a related result. 

Before we state our result on permutations, we recall that 
an ideal $\mathcal{I}$ is said to be a \textbf{simple density ideal} if there exists a function $g: \mathbf{N} \to [0,\infty)$ such that $\lim_n g(n)=+\infty$, $\lim_n n/g(n)\neq 0$, and 
$
\textstyle \mathcal{I}=\mathcal{Z}_g:=\left\{S\subseteq \mathbf{N}: \mathsf{d}^\star_g(S)=0\right\},
$ 
where $\mathsf{d}^\star_g: \mathcal{P}(\mathbf{N})\to [0,\infty]$ is the function defined by
\begin{equation}\label{eq:upperdensity}
\forall S\subseteq \mathbf{N},\quad
\mathsf{d}^\star_g(S)=\limsup_{n\to \infty}
\frac{|S\cap [1,n]|}{g(n)}.
\end{equation}
In particular, $\mathcal{Z}$ is the simple density ideal generated by $g(n)=n$, cf. \cite{MR3391516, MR3771234}. 
Note that the simple density ideal $\mathcal{Z}_g$ for which $n/g(n)$ is bounded does not necessarily concide with $\mathcal{Z}$. Indeed, set $S_k:=[(2k)!,(2k+1)!)$ for all $k \in \mathbf{N}$ and $S:=\bigcup_k S_{2k}$. Define $g(n)=n^2$ if $n \in S$ and $g(n)=n$ otherwise. Then it is easy to see that $S \in \mathcal{Z}_g\setminus \mathcal{Z}$. Hence $\mathcal{Z}$ is properly contained in $\mathcal{Z}_g$.

Lastly, given an ideal $\mathcal{I}$ and a sequence $x \in \omega$, 
we say that $\eta \in\mathbf{R}$ is an $\mathcal{I}$-limit point of $x$ if 
 there exists a subsequence $(x_{n_k})$ for which $\lim_k x_{n_k}=\eta$ and $\{n_k: k \in \mathbf{N}\} \notin \mathcal{I}$, cf. \cite{MR1181163}. 

The following result has been essentially proved for the case $\mathcal{I}=\mathcal{J}=\mathcal{Z}$ in \cite{Connor2020}.

\begin{thm}\label{thm:mainperm}
Let $\sigma$ be a permutation of $\mathbf{N}$ and let $\mathcal{I}, \mathcal{J}$ be ideals on $\mathbf{N}$ such that $\mathcal{J}$ is not maximal. 
Then the following are equivalent\textup{:}
\begin{enumerate}[label={\rm (\textsc{P}\arabic{*})}]
\item \label{item:P1a} $\sigma \in (c(\mathcal{I}), c(\mathcal{J}))$ 
and  $\mathcal{I}\text{-}\lim x=\mathcal{J}\text{-}\lim \sigma(x)$ for all $x \in c(\mathcal{I})$\textup{;}
\item \label{item:P2a} $\sigma \in (c(\mathcal{I}), c(\mathcal{J}))$\textup{;}
\item \label{item:P1b} $\sigma \in (c_0(\mathcal{I}), c_0(\mathcal{J}))$\textup{;}
\item \label{item:P2b} $\sigma \in (c_0(\mathcal{I}), c(\mathcal{J}))$\textup{;}
\item \label{item:P3} $\sigma(\mathcal{I}) \subseteq \mathcal{J}$\textup{;}
\item \label{item:reg} $\sigma$ is $(\mathcal{I},\mathcal{J})$-regular\textup{;}
\item \label{item:c0c0} $\sigma \in (c_0(\mathcal{I}) \cap \ell_\infty, c_0(\mathcal{J})\cap \ell_\infty)$\textup{.}
\end{enumerate}

In addition, if $\mathcal{I}=\mathcal{Z}_{g}$ and $\mathcal{J}=\mathcal{Z}_{h}$ are two simple density ideals such that
\begin{equation}\label{eq:conditiong1g2}
\textstyle 
\forall \alpha>1, \quad 
\limsup_{n\to \infty} \frac{n}{g(n)}<\infty 
\quad \text{ and }\quad 
\limsup_{n\to \infty} \frac{g(\lfloor \alpha n\rfloor)}{h(n)}<\infty,
\end{equation}
then the above conditions are equivalent to\textup{:}
\begin{enumerate}[label={\rm (\textsc{P}\arabic{*})}]
\setcounter{enumi}{7}
\item \label{item:P4} 
$0$ is not a $\mathcal{J}$-limit point of $\widehat{\sigma}$\textup{.}
\end{enumerate}
\end{thm}
\begin{proof}
\ref{item:P1a} $\implies$ \ref{item:P2a} $\implies$ \ref{item:P2b} and \ref{item:P1a} $\implies$ \ref{item:P1b} $\implies$ \ref{item:P2b} are obvious.

\ref{item:P2b} $\implies$ \ref{item:P3} Suppose for the sake of contradiction that there exists $S\in  \sigma(\mathcal{I})\setminus \mathcal{J}$. Then $\bm{1}_{\sigma^{-1}(S)} \in c_0(\mathcal{I})$ and $\sigma(\bm{1}_{\sigma^{-1}(S)})=\bm{1}_{S} \in c(\mathcal{J})$, so that $S \in \mathcal{J}^\star$. Since $\mathcal{J}$ is not 
a maximal ideal, there exists a set $T\notin \mathcal{J}\cup \mathcal{J}^\star$, so that also $W:=T\cap S\notin \mathcal{J}\cup \mathcal{J}^\star$. 
Then $\sigma^{-1}(W)\subseteq \sigma^{-1}(S) \in \mathcal{I}$, which implies $\bm{1}_{\sigma^{-1}(W)} \in c_0(\mathcal{I})$ and $\sigma(\bm{1}_{\sigma^{-1}(W)})=\bm{1}_{W} \notin c(\mathcal{J})$.

\ref{item:P3} $\implies$ \ref{item:P1a} Fix $x \in c(\mathcal{I})$ such that $\mathcal{I}\text{-}\lim x=\eta$. Given $\varepsilon>0$, set $S:=\{n\in \mathbf{N}: |x_n-\eta|>\varepsilon\}$. Then $S \in \mathcal{I}$ and 
$\{\sigma^{-1}(n)\in \mathbf{N}: |x_{\sigma^{-1}(n)}-\eta|>\varepsilon\}=\sigma(S) \in \mathcal{J}$.  
Since $\varepsilon$ is arbitrary, it follows that 
$\sigma(x) \in c(\mathcal{J})$ and 
$\mathcal{J}\text{-}\lim \sigma(x)=\eta$. 

\ref{item:P3} $\Longleftrightarrow$ \ref{item:reg} $\Longleftrightarrow$ \ref{item:c0c0} 
Note that if $\sigma$ is a permutation of $\mathbf{N}$ then conditions \ref{item:T1} and \ref{item:T2} are trivially verified for the matrix $A_\sigma$. Also observe that, for any set $E\subset\mathbf{N}$  one has that $A_\sigma(\bm{1}_E)=\bm{1}_{\sigma(E)}$.
It is clear that \ref{item:reg} implies \ref{item:c0c0}. If $\sigma$ satisfies \ref{item:c0c0}, then $E\in\mathcal{I}$ implies that $\sigma(E)\in\mathcal{J}$, which, in turn, yields \ref{item:P3}. Next we establish that if \ref{item:P3} holds, then \ref{item:T3} holds for $A_\sigma$; this will yield, via Theorem \ref{thm:mainsilvermantoeplitz}, that  \ref{item:reg} holds. 
To see this, fix $E\in\mathcal{I}$ and observe that  $\sum_{k\in E}a_{n,k}=\sigma (\bm{1}_E)_n$ for all $n\in\mathbf{N}$
 and $\bm{1}_{\sigma(E)}\in c_0 (\mathcal{J})\cap\ell_\infty$, i.e., $\mathcal{J}$-$\lim_n \sum_{k\in E}|a_{n,k}|=0$.

%

\bigskip

For the second part of the statement, assume that $\mathcal{I}$ is the simple density ideal $\mathcal{Z}_g$ such that $n/g(n)$ is bounded sequence which, thanks to \cite[Proposition 1]{MR3771234}, is equivalent to $\mathcal{Z}\subseteq \mathcal{Z}_g$.

\ref{item:P3} $\implies$ \ref{item:P4} Suppose that 
$0$ is a $\mathcal{J}$-limit point of $\widehat{\sigma}$, 
hence there exists a subsequence $(\widehat{\sigma}_{n_k})$ such that $\lim_k \widehat{\sigma}_{n_k}=0$ and $S:=\{n_k: k \in \mathbf{N}\} \notin \mathcal{J}$. To conclude, we claim that $S \in \sigma(\mathcal{I})$, that is, $W:=\{\sigma^{-1}(n_k): k \in \mathbf{N}\} \in \mathcal{I}$. Since $\lim_k n_k/\sigma^{-1}(n_k)=0$, we conclude that for all $c>0$ there exists $k_0$ such that $\sigma^{-1}(n_k) \ge cn_k$ for all $k\ge k_0$. It follows that $W \in \mathcal{Z}\subseteq \mathcal{I}$.

\ref{item:P4} $\implies$ \ref{item:P3} 
It is not difficult to see that $0$ is not a $\mathcal{J}$-statistical limit point of $\widehat{\sigma}$ if and only if 
$$
\lim_{\varepsilon \to 0^+}\mathsf{d}^\star_h(E_\varepsilon)=0, 
\quad \text{where} \quad  
E_\varepsilon:=\{n \in \mathbf{N}: \widehat{\sigma}_n<\varepsilon\},
$$
cf. \cite[Theorem 2.2]{MR3883171} or \cite[Theorem 18]{MR3405547}. At this point, fix $S \subseteq \mathbf{N}$ such that $\sigma(S) \notin \mathcal{J}$. 
We will show that $S \notin \mathcal{I}$. 
Note that if $n \in [1,m]\setminus E_\varepsilon$ then $\sigma^{-1}(n)\le n/\varepsilon \le m/\varepsilon$. Since
$$
\forall \varepsilon>0, \forall m \in \mathbf{N},\quad 
\sigma^{-1}(\sigma(S) \cap [1,m]\setminus E_\varepsilon) \subseteq 
S \cap [1,m/\varepsilon],
$$
we obtain that
$$
\forall \varepsilon>0, \forall m \in \mathbf{N},\quad 
 \frac{|\sigma(S) \cap [1,m]\setminus E_\varepsilon|}{h(m)} \le 
\frac{|S \cap [1,m/\varepsilon]|}{g(\lfloor m/\varepsilon\rfloor)}\cdot \frac{g(\lfloor m/\varepsilon\rfloor)}{h(m)},
$$
Considering that the function $\mathsf{d}_h^\star$ is monotone and subadditive, we have that
$$
\mathsf{d}_h^\star(\sigma(S)\setminus E_\varepsilon) \ge \mathsf{d}_h^\star(\sigma(S))-\mathsf{d}_h^\star(\sigma(S)\cap E_\varepsilon) \ge \mathsf{d}_h^\star(\sigma(S))-\mathsf{d}_h^\star(E_\varepsilon),
$$
hence there exists $\varepsilon_0>0$ such that 
$\mathsf{d}_h^\star(\sigma(S)\setminus E_{\varepsilon_0})>0$. 
To conclude the proof, since 
$$
c:=\limsup_{m\to \infty} \frac{g(\lfloor m/\varepsilon_0\rfloor)}{h(m)}<\infty
$$
by the the hypothesis \eqref{eq:conditiong1g2}, we obtain that
$$
0<\mathsf{d}_h^\star(\sigma(S)\setminus E_{\varepsilon_0}) \le c\cdot  \limsup_{m\to \infty}\frac{|S \cap [1,m/\varepsilon_0]|}{g(\lfloor m/\varepsilon_0\rfloor )} \le c\cdot \mathsf{d}_g^\star(S),
$$
which implies that $S \notin \mathcal{I}$.
\end{proof}


\section{Concluding Remarks and Open Questions}

It follows by Theorem \ref{thm:mainperm} that, if $\mathcal{J}$ is not a maximal ideal, 
then a permutation matrix belongs to $(c_0(\mathcal{I}), c_0(\mathcal{J}))$ if and only if it belongs to $(c_0(\mathcal{I}) \cap \ell_\infty, c_0(\mathcal{J}) \cap \ell_\infty)$. 
However, this does not hold in general. Indeed, the characterization provided in Theorem \ref{thm:mainc0c0} does not hold for its unbounded analogue $(c_0(\mathcal{I}), c_0(\mathcal{J}))$, that we leave as an open question for the interested reader. 
To this aim, consider the following example.  Let $A=(a_{n,k})$ be the matrix defined by $a_{n,n}=n$ if $n \in S:=\{k^2: k \in \mathbf{N}\}$, $a_{n,n}=1$ if $n\notin S$, and $a_{n,k}=0$ otherwise. Then it is easy to see that $A \in (c_0(\mathcal{Z}), c_0(\mathcal{Z}))$. On the other hand, if $x=\bm{1}_S$, then $Ax \notin \ell_\infty$, which proves that $A \notin (c_0(\mathcal{Z})\cap \ell_\infty, c_0(\mathcal{Z}) \cap \ell_\infty)$. This implies that condition \ref{item:T1} is not necessary for a characterization of the class $(c_0(\mathcal{I}), c_0(\mathcal{J}))$.

Lastly, we also leave as an open question to characterize the class of bounded [resp. unbounded] $(\mathcal{I}, \mathcal{J})$-conservative matrices, that is, the set of matrices $A \in (c(\mathcal{I}) \cap \ell_\infty, c(\mathcal{J}) \cap \ell_\infty)$ [resp., $A \in (c(\mathcal{I}) , c(\mathcal{J}))$] which do not necessarily preserve the corresponding ideal limits.

\subsection*{Acknowledgments} 
P.L. is grateful to PRIN 2017 (grant 2017CY2NCA) for financial support.

\bibliographystyle{amsplain}

\end{document}